\documentclass[
final,nomarks
]{dmtcs-episciences}

\usepackage[utf8]{inputenc}
\usepackage{amsmath,amssymb,amsfonts,amsthm}
\usepackage{subfigure}

\usepackage[round]{natbib}

\newcommand{\B}{{\mathcal B}}
\newcommand{\G}{{\mathcal G}}
\newcommand{\Set}{{\mathcal S}}
\newcommand{\K}{{\mathcal K}}
\newcommand{\R}{{\mathcal R}}
\newcommand{\Ham}{{\mathcal H}}
\newcommand{\whp}{with high probability}
\newcommand{\V}{\mathcal{V}}
\newcommand{\W}{\mathcal{W}}
\newcommand{\N}{{\mathcal N}}
\newcommand{\Po}[1]{\textrm{Po}\left(#1\right)}
\newcommand{\Bin}[2]{\textrm{Bin}\left(#1,#2\right)}
\newcommand{\E}{\mathbb{E}}
\newcommand{\Pra}[1]{\Pr\left\{#1\right\}}
\newcommand{\PraG}[1]{\text{Pr}_{\G}\left\{#1\right\}}
\newcommand{\Bnmp}{\mathcal{B}\left(n,m,p\right)}
\newcommand{\Xnmp}{X_q\left(n,m,p\right)}
\newcommand{\END}[1]{\text{END}(#1)}
\newcommand{\LARGEv}{\text{LARGE}}
\newcommand{\SMALLv}{\text{SMALL}}
\newcommand{\eps}{\varepsilon}
\newcommand{\dO}{d_0}
\newcommand{\djeden}{d_1}

\theoremstyle{plain}
\newtheorem{thm}{Theorem}

\newtheorem{lem}{Lemma}

\theoremstyle{definition}
\newtheorem{rem}{Remark}

\author{Katarzyna Rybarczyk
	\thanks{supported by NCN (Narodowe Centrum Nauki) grant 2014/13/D/ST1/01175}}
\title{Finding Hamilton cycles in random intersection graphs}
\affiliation{
  Adam Mickiewicz University, Pozna\'n, Poland}

\received{2017-2-15}

\revised{2017-10-5}

\accepted{2018-1-22}

\begin{document}
\publicationdetails{20}{2018}{1}{8}{3144}
\maketitle
\begin{abstract}
  The construction of the random intersection graph model is based on a random family of sets.  Such structures, which are derived from intersections of sets, appear in a natural manner in many applications. In this article we study the problem of finding a Hamilton cycle in a random intersection graph. To this end we analyse a classical algorithm for finding Hamilton cycles in random graphs (algorithm HAM) and study its efficiency on graphs from a family of random intersection graphs (denoted here by $\mathcal{G}\left(n,m,p\right)$). We prove that the threshold function for the property of HAM constructing a Hamilton cycle in $\mathcal{G}\left(n,m,p\right)$ is the same as the threshold function for the minimum degree at least two. Until now, known algorithms for finding Hamilton cycles in $\mathcal{G}\left(n,m,p\right)$ were designed to work in very small ranges of parameters and, unlike HAM, used the structure of the family of random sets.
\end{abstract}

\section{Introduction}\label{Introduction}

Since its introduction by  \cite{GpSubgraph} the random intersection graph model and its generalisations have proven to have many applications. To mention just a few: ``gate matrix layout'' for VLSI design (see e.g. \cite{GpSubgraph}), cluster analysis and classification (see e.g. \cite{RIGGodehardt1}), analysis of complex networks (see e.g. \cite{RIGClustering2, RIGTunableDegree}), secure wireless networks (see e.g. \cite{WSNphase2}) or epidemics (\cite{GpEpidemics}). For more details we refer the reader to survey papers \cite{RIGsurvey1,RIGsurvey2}.   

In the random intersection graph model $\mathcal{G}\left(n,m,p\right)$ to each vertex from the vertex set $\V$ $(|\V|=n)$ we assign a random set of its features $\W(v)$ from an auxiliary set $\W$ $(|\W|=m(n))$. For each $w\in \W$ and $v\in\V$ we have $w\in \W(v)$ with probability $p$, $p=p(n)\in (0,1)$, independently of all other elements from $\V$ and $\W$. We connect vertices $v$ and $v'$ by an edge if sets $\W(v)$ and $\W(v')$ intersect.

In the course of last years various properties of random intersection graph were studied. However, little is still known about algorithms which might efficiently construct or find structures in random intersection graphs. 
In this article we address the problem of efficiently finding a Hamilton cycle in $\mathcal{G}\left(n,m,p\right)$. The problem of determining a threshold function for the property of having a Hamilton cycle in the random intersection graph has already been studied by several authors, for example: \cite{SpirakisHamiltonCycles},  \cite{GpCoupling} (both for the model considered in this paper), and  \cite{UniformHamilton1} (for the uniform random intersection graph model). 
Here we analyse algorithmic aspects of the problem of finding a Hamilton cycle in $\mathcal{G}\left(n,m,p\right)$. \cite{SpirakisHamiltonAlgorytm}  proposed a randomised algorithm which with probability tending to 1 as $n\to\infty$ in polynomial time finds a Hamilton cycle in $\mathcal{G}\left(n,m,p\right)$ if $m/(n\ln n)\to 0$ and $p=\frac{\ln n +\omega(n)}{m}$ with $\omega(n)$ tending slowly to infinity. Our aim is to show that there exists an algorithm which efficiently finds a Hamilton cycle in $\mathcal{G}\left(n,m,p\right)$ for a wider range of parameters $n,m,p$. To this end we take algorithm  HAM introduced by  \cite{HamiltonAlg1} and analyse its performance on $\mathcal{G}\left(n,m,p\right)$. 

All limits in the paper are taken as $n\rightarrow \infty$.
Throughout the paper we use standard asymptotic notation
$a_n=o(b_n)$ if $a_n/b_n\to 0$ as $n\to\infty$ and $a_n=O(b_n)$ if there exists a constant $C$ such that $|a_n|\le C|b_n|$, for large $n$. By $\Bin{n}{p}$ we denote the binomial distribution with parameters $n$, $p$. We also use the phrase ``{\whp}'' to say with probability tending to one as $n$ tends to infinity. We consequently omit $\lfloor\cdot\rfloor$ and $\lceil\cdot\rceil$ for the sake of clarity of presentation. All inequalities hold for $n$ large enough.

\section{Main results}

Algorithm HAM, as presented by \cite{HamiltonAlg1}, is designed to search for a Hamilton cycle in any graph with the minimum degree at least $2$. It is assumed that the vertex set of the input graph   is ordered, i.e. $\V=\{v_1,v_2,\ldots,v_n\}$ (for simplicity we write $v_i<v_j$ if $i<j$). Algorithm HAM utilises the rotation technique introduced by \cite{ERHam1}. The rotation technique uses the fact that, given a path $P$ of length $k$ in a graph $\G$ (i.e. $P=(u_1,\ldots,u_k)$), if there is an edge $e=\{u_k,u_i\}$, $2\le i\le k-2$, we may construct another path of length $k$ denoted by
\begin{displaymath}
\text{ROTATE}(P, e) = (u_1,u_2,\ldots,u_i,u_k,u_{k-1},\ldots,u_{i+1}).
\end{displaymath}
\begin{figure}[p]
	\includegraphics[height=0.95\textheight]{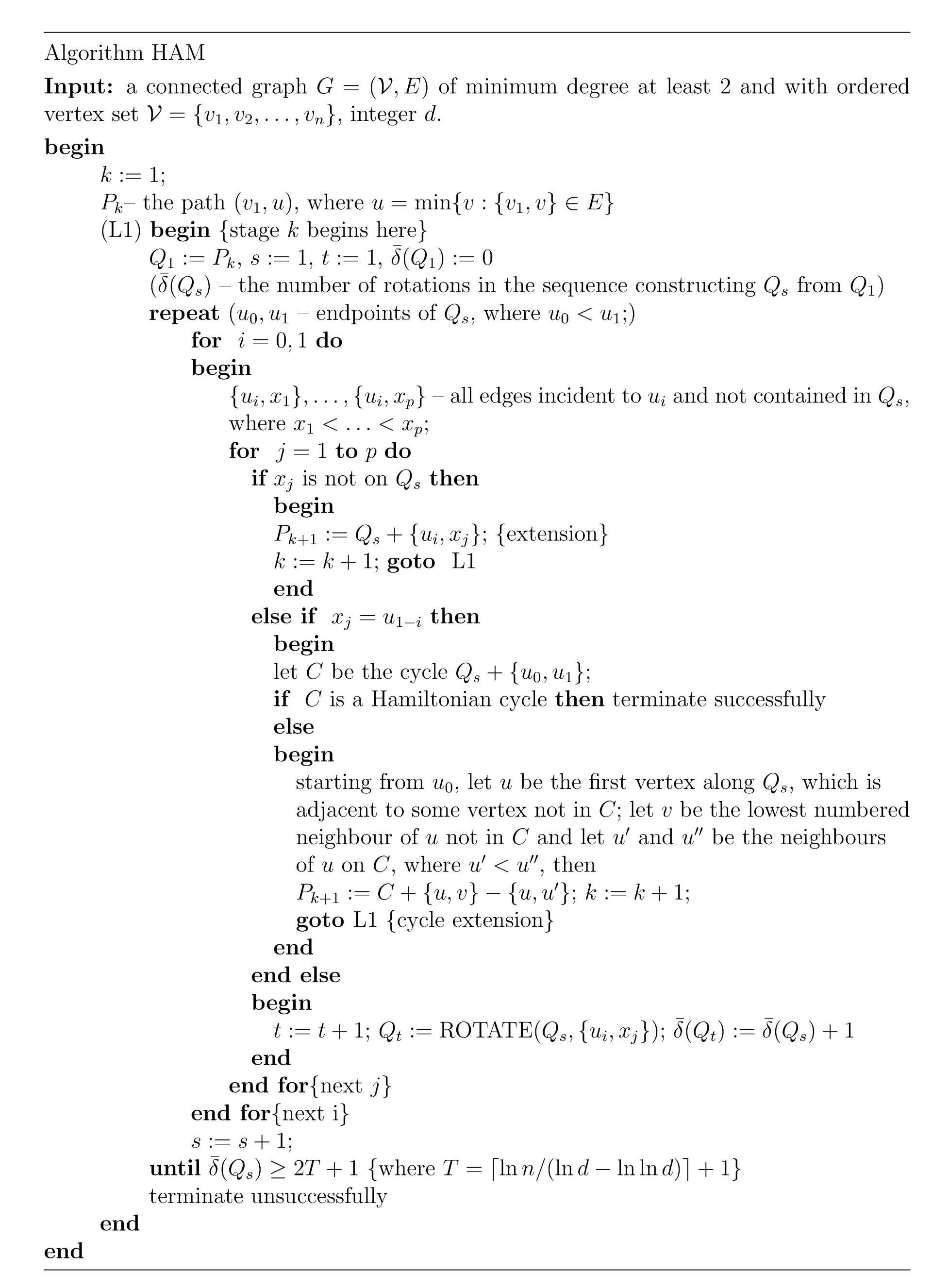}
	\caption{Algorithm HAM introduced by \cite{HamiltonAlg1}}\label{FigureHAM}
\end{figure}   
We state algorithm HAM following \cite{HamiltonAlg1} (see Figure~\ref{FigureHAM}). In most cases we keep notation consistent with the one introduced by \cite{HamiltonAlg1}.
Algorithm HAM, starting with a path consisting of one vertex $v_1$, in each step  extends the path (if its end has a neighbour outside the path or both endpoints are neighbours) or, if extension is not possible, it searches for new paths of the same length using P\'{o}sa's rotation technique. In the latter case the algorithm explores all possible rotations in the BFS--type manner, i.e. given a considered path $P=(u_1,\ldots,u_k)$ of length $k$, it explores all neighbours of $u_k$ and $u_1$ and new paths resulting from rotations related to those neighbours. Then for each new path, one by one, HAM extends the path or, if an extension is not possible, it explores neighbours of the ends of this path and does rotations. The algorithm stops if either it finds a Hamilton cycle or it is not able to extend the path of length $k$. More precisely, it stops at stage $k$ when it has explored without extension all paths which result from at most $2T+1$ ($T$ is a function of $n$ and $d$) rotations of the initial path of length $k$.

In what follows whenever HAM is executed on $\mathcal{G}\left(n,m,p\right)$ with $mp^2\le 1$ we set
\begin{displaymath}
d=nmp^2.
\end{displaymath}

The main result of this article is the following.

\begin{thm}\label{TwierdzenieHamiltonGp}
	Let $\mathcal{G}\left(n,m,p\right)$ be a random intersection graph with $\ln n=o(m)$ and $mp^2\le 1$, then
	\begin{displaymath}
	\lim_{n\to\infty}\Pra{\{\delta(\mathcal{G}\left(n,m,p\right))\ge 2\}\cap\{\text{HAM terminates unsuccessfully on }\mathcal{G}\left(n,m,p\right)\}}=0.
	\end{displaymath} 
\end{thm}
We  rigorously prove the above theorem in the case $m\ge n^{24/25}$. Then we present a sketch of the proof in the latter case.  
\begin{rem}
	We give the main theorem only for $\ln n=o(m)$ and $mp^2\le 1$. The latter cases would require considering more cases in the proofs however the reasoning would be easier. Namely, if $m=O(\ln n)$, in the relevant range of parameters, {\whp} $\mathcal{G}\left(n,m,p\right)$ consists of $m=O(\ln n)$ independent very large cliques thus it is very close to a complete graph. Similarly if $mp^2>1$ then the probability of given two vertices being connected by an edge is a constant and {\whp} almost all vertices have degree at least $n/1000$. This case would require, among others, setting $d=n$, in some places replacing  $d_0$ by $n$, setting $T=2$, and redefining property {\bf P3} defined in the proof (as in this case $n/d_1<1$). Even though this case would require changes to some parts of the proof, in general the reasoning would follow the same lies and would be easier. Thus we do not include it for shortness.    	
\end{rem}
\begin{rem}\label{RemComplexity}
	Using the same arguments as \cite{HamiltonAlg1}, we may show that  {\whp} the time complexity of HAM on $\mathcal{G}\left(n,m,p\right)$ is $O(n^{4+\varepsilon})$ (for any $\eps>0$). For completeness we give the proof in Appendix~A. 
\end{rem}	

Theorem~\ref{TwierdzenieHamiltonGp} together with the probability of the event that the minimum degree of $\mathcal{G}\left(n,m,p\right)$ is at least 2, $\delta(\mathcal{G}\left(n,m,p\right))\ge 2$,  gives the probability of the property that HAM finds a Hamilton cycle in $\mathcal{G}\left(n,m,p\right)$. The proof of Theorem~\ref{LemmaDegree} is presented in Appendix B. 
\begin{thm}\label{LemmaDegree} 
	If $\ln^2 n=o(m)$
	$$p(1-(1-p)^{n-1})=\frac{\ln n + \ln a_n + c_n}{m},$$
	where $c_n=o(\ln n)$, and for some $\eps>0$
	$$a_n = 
	\begin{cases}
	1&\text{for }m<(1-\eps)\frac{n\ln n}{\ln\ln n};\\
	\frac{np\ln n}{e^{np}-1}&\text{for }m>(1+\eps)\frac{n\ln n}{\ln\ln n},\\
	\end{cases}
	$$
	then
	$$
	\lim_{n\to\infty}\Pra{\delta(\mathcal{G}\left(n,m,p\right))\ge 2}=
	\begin{cases}
	0&\text{ for }c_n\to -\infty;\\
	e^{-e^{-c}}&\text{ for }c_n\to c\in (-\infty,\infty);\\
	1&\text{ for }c_n\to \infty.
	\end{cases}
	$$
\end{thm}

In the above mentioned theorem we exclude the case $m = (1+o(1))n\ln n/\ln\ln n$. This case may be analysed by the same methods as those presented in the proofs but would require several additional case studies and more cases in the definition of $a_n$. We omit it for the sake of clarity of presentation. The result may be also expanded to the case $c_n=\Omega(\ln n)$ by a simple coupling and monotonicity of the property $\{\delta(\mathcal{G}\left(n,m,p\right))\ge 2\}$.

The remaining part of the article is organized as follows. In Section~\ref{SectionNotation} we introduce some notation and gather facts which will be used repeatedly in the whole article. In the following three         sections we present a proof of Theorem~\ref{TwierdzenieHamiltonGp} in the case $mp^2\le 1$ and  $m\ge n^{1-\eps}$ for $\eps = 1/25$. Namely, in Section~\ref{SectionProperties} we study the properties of $\mathcal{G}\left(n,m,p\right)$, which are crucial for the analysis of algorithm HAM on $\mathcal{G}\left(n,m,p\right)$. Then in Section~\ref{SectionDeletable} we discuss the notion of deletable sets. In Section~\ref{SectionProof} we prove Theorem~\ref{TwierdzenieHamiltonGp} in the case $mp^2\le 1$ and $m\ge n^{1-\eps}$   using facts proved in the first part of the article. In Section~\ref{SectionSmallm}  we give a sketch of the proof in the case where $m<n^{1-\eps}$.

\section{Notation and important facts}\label{SectionNotation}

For any graph $\G$ we denote by $E(\G)$ its edge set. If we mention an intersection graph $\G$ we mean not only the graph but also the underlying structure of feature sets attributed to vertices. For any intersection graph $\G$ with a vertex set $\V$ and feature set $\W$, if $w\in\W(v)$ then we say that vertex $v$ \emph{chose} feature $w$ and $w$ \emph{was chosen by} $v$.
For $v\in \V$, $w\in \W$,  $\Set\subseteq \V$, and $\R\subseteq \W$ in random intersection graph $\G$
we use the following notation
\begin{align*}
\W_{\G}(v)&=\{w\in\W: \text{$v$ chose $w$}\},
&\W_{\G}(\Set)&=\bigcup_{v\in \Set}\W_{\G}(v);
\intertext{{\it i.e.} the set of features chosen by $v$ (and, respectively, chosen by vertices from $\Set$);}
\V_{\G}(w)&=\{v\in\V: \text{$w$ was chosen by $v$}\},
&\V_{\G}(\R)&=\bigcup_{w\in \R}\V_{\G}(w);
\intertext{{\it i.e.} the set of vertices which chose $w$ (and, respectively, chose features from $\R$);}
\W_{\G}'(v)&=\{w\in\W(v): |\V_{\G}(w)|\ge 2\},
&\W_{\G}'(\Set)&=\bigcup_{v\in \Set}\W'_{\G}(v);
\intertext{{\it i.e.} the set of features chosen by  $v$ (and, respectively, chosen by vertices from $\Set$) which contribute to at least one edge in $\G$;}
\W_{\G}''(\Set)&=\{w\in\W(\Set): |\V_{\G}(w)\cap \Set|\ge 2\};
\intertext{{\it i.e.} the set of features chosen by at least two vertices from $\Set$.}
\intertext{Moreover for any graph $\G$ with vertex set $\V$ and edge set $E(\G)$ and for any $\Set\subseteq\V$ we set } 
\N_{\G}(v)&=\{v'\in\V\setminus\{v\}: \{v,v'\}\in E(\G)\},
&
\N_{\G}(\Set)&=\bigcup_{v\in \Set}\N_{\G}(v)\setminus \Set;
\end{align*}
{\it i.e.} the set of neighbours  of  $v$ in $\G$  (and, respectively, neighbours  of  $\Set$ in $\G$). 

In addition 
\begin{align*}
W_{\G}(v)&=|\W_{\G}(v)|,&W_{\G}(\Set)&=|\W_{\G}(\Set)|,\\
V_{\G}(w)&=|\V(w)|,&V_{\G}(\R)&=|\V_{\G}(\R)|,\\
W_{\G}'(v)&=|\W_{\G}'(v)|,&W_{\G}'(\Set)&=|\W_{\G}'(\Set)|,\\
W_{\G}''(\Set)&=|\W_{\G}''(\Set)|,&&\\
N_{\G}(v)&=|\N_{\G}(v)|,&N_{\G}(\Set)&=|\N_{\G}(\Set)|,
\end{align*}
and $\deg(v)=\deg_{\G}(v)=N_{\G}(v)$.

In the above mentioned notation we consequently omit subscript $\G$ when it is clear from the context which $\G$ we have in mind. 

Alternatively, in order to obtain an instance $\G$ of $\mathcal{G}\left(n,m,p\right)$, we may first construct an instance $\B$ of a random bipartite graph $\Bnmp$ with bipartition $(\V,\W)$ in which each edge $\{v,w\}$, $v\in\V$ and $w\in\W$, appears independently with probability $p$.  We connect vertices $v\in \V$ and $v'\in\V$ by an edge in $\mathcal{G}\left(n,m,p\right)$ if they have a common neighbour in $\Bnmp$. An instance $\G$ of $\mathcal{G}\left(n,m,p\right)$ constructed in this manner from $\B$ we will call \emph{associated with} $\B$. Note that $\W(v)$, $v\in\V$, is the set of neighbours of vertex $v$  in $\Bnmp$. Moreover $\V(w)$ is the set of neighbours of $w\in\W$ in $\Bnmp$. In addition $W(v)$ and $V(w)$ are degrees of $v\in\V$ and $w\in\W$ in $\Bnmp$.

Moreover, let 

\begin{equation*}
\dO=mp(1-(1-p)^{n-1}),\quad\quad\djeden=nmp^2.
\end{equation*}
Note that $\dO$ is the expected value of $W'_{\mathcal{G}\left(n,m,p\right)}(v)$ and $\djeden$ is almost the expected degree of a vertex in $\mathcal{G}\left(n,m,p\right)$.

The following relations between $d_0$, $\djeden$, and $mp$ will be frequently used without mentioning them directly.  
In the range of parameters we are interested in ({\it i.e.} $n$, $m$, and $p$ such that the probability of event   $\{\delta(\mathcal{G}\left(n,m,p\right))\ge 2\}$ does not tend to $0$), we have 
$$\dO\ge (1+o(1))\ln n .$$
For the proof  see for example Lemma 5.1 in \cite{RIGcoupling2}.

Moreover by definition
\begin{align*}
\dO&\le mp,\\
\dO&\le nmp^2 = \djeden,\\
\dO&\le mp\min\{np,1\},\\
\djeden&\le 2\dO\max\{np,1\}.
\end{align*}
In addition,
\begin{itemize}
	\item[] if $np > 40$ then $0.9\, mp\le \dO \le mp$ and
	\item[] if $np\le 40$ then $\dO \le \djeden\le 40.1 \dO$. 
\end{itemize}
To get the first inequality for $np>40$ we note that 
$$d_0\ge mp(1-e^{-(n-1)p})\ge mp(1-e^{-39})\ge 0.9mp.$$
In the proof of the second inequality for $np\le 40$ we have, for large $n$, 
$$
40.1 d_0\ge 40.1 d_1 \cdot \frac{n-1}{n}\cdot \frac{1-e^{-(n-1)p}}{(n-1)p}\ge (1+o(1)) 40.1 d_1 \frac{1-e^{-40}}{40}\ge d_1
$$
since $f(x)=(1-e^{-x})/x$ is decreasing for $x > 0$. 

We will frequently use Chernoff's inequality  in the following form (see for example Theorem 2.1 in \cite{KsiazkaJLR}). 

\begin{lem}
	Let $X$ be a random variable with the binomial distribution and expected value $\mu$, then for any $0<\eps<1$ 
	$$
	\Pra{X\le \eps \mu}\le \exp\left(-\psi(\eps)\mu\right),
	$$
	where
	$$
	\psi(\eps)=\eps\ln \eps + 1 - \eps. 
	$$
\end{lem}

\section{Graph properties}\label{SectionProperties}

In this and the following section we assume that $m\ge n^{1-\eps}$ (for $\eps=1/25$) and $mp^2\le 1$ ({\it i.e.} $n\ge \djeden$).

The efficiency of algorithm HAM on a random graph is due to the fact that random graphs, in general,  have good expansion properties. Below we list the properties which will be used in the analysis of algorithm HAM on $\mathcal{G}\left(n,m,p\right)$.   

Given an intersection graph $\G$ associated with a bipartite graph $\B$, integers  $\dO=\dO(n)$, $\djeden=\djeden(n)$, $m=m(n)$, a real $p=p(n)\in (0,1)$, sets $\SMALLv\subseteq \V$ and $\LARGEv=\V\setminus\SMALLv$, and a constant $b_1>0$  we define the following properties:
\begin{itemize}
	\item[{\bf P0}] $\delta(\G)\ge 2$;
	\item[{\bf P1}]  $|\SMALLv|\le n^{1/3}$ and \\
	if $d_0\ge 2\ln n$ then $|\SMALLv|=\emptyset$;
	\item[{\bf P2}] There are no $v,v'\in \SMALLv$ at distance 4 or less apart;
	\item [{\bf P3}] 		
	For all $\Set\subseteq \LARGEv$, if $|\Set|\le n/\djeden$ then $N(\Set)\ge b_1 \djeden |\Set|$. 
	\item [{\bf P4}] 		
	For all $v\in\V$ 
	$$W(v)\le 4mp,\quad W'(v)\le 4\dO,\quad N(v)\le 12\djeden.$$
	\item [{\bf P5}] For all $w\in\W$ 
	$$V(w)\le \frac{\ln n}{\ln\ln n} \max\{np,4\}.$$
\end{itemize}

\begin{lem}\label{LemProperties}
	Let $\mathcal{G}\left(n,m,p\right)$ be a random intersection graph with $p\le \sqrt{(3\ln n)/m}$ and
	$$
	\dO\ge (1+o(1))\ln n.
	$$
	Moreover let 
	$$
	\SMALLv=\{v\in \V: W'(v)\le 0.1\dO\},\quad \LARGEv=\{v\in \V: W'(v) > 0.1\dO\}.
	$$
	If $m\ge n^{1-\eps}$ $(\eps=1/25)$ then {\whp} $\mathcal{G}\left(n,m,p\right)$ has properties {\bf P1}--{\bf P5} with $b_1=0.001$.
\end{lem}
It is easy to prove by the first moment method that if $p>\sqrt{3\ln n/m}$ then {\whp} $\mathcal{G}\left(n,m,p\right)$ is a complete graph. Therefore if $p>\sqrt{3\ln n/m}$ then {\whp} HAM finds a Hamilton cycle in $\mathcal{G}\left(n,m,p\right)$ without any problem. This is why we may restrict ourselves to the case $p\le \sqrt{3\ln n/m}$.    

\pagebreak

\begin{proof}[ of Lemma~\ref{LemProperties}]
	Recall that all inequalities hold for large $n$.
	
	In the proofs we assume that $\ln n=o(m)$. If we make some additional assumptions concerning relation between $n$ and $m$, we mention it at the beginning of the proof of the considered property. Moreover, we set $b_1=0.001$.

	\noindent{\bf P1}
	Note that $W'(v)$ has the binomial distribution $\Bin{m}{\dO/m}$. Therefore for $\dO\ge (1-o(1))\ln n$ by Chernoff's inequality
	\begin{align*}
	\E|\SMALLv|
	&\le n\Pra{W'(v)\le 0.1\cdot \dO}\\
	&\le n \exp\left(-\psi(0.1)\cdot \dO\right)\\
	&\le \exp\left(\ln n-0.669\ln n\right)= o( n^{1/3}).
	\end{align*}
	Thus by Markov's inequality we get that {\whp} $|\SMALLv|\le n^{1/3}$. Moreover the same calculation implies that if $\dO\ge 2\ln n$ then $\E|\SMALLv|=o(1)$ {\it i.e.} {\whp} $\SMALLv=\emptyset$.

	\noindent {\bf P2} In the proof of {\bf P2} we assume additionally  that $m\ge n^{1-\eps}$ (for $\eps=1/25$). First recall that if $\dO\ge 2\ln n$ then {\whp} $\SMALLv=\emptyset$, {\it i.e.} the statement holds trivially. 
	
	Now assume that $\dO\le 2\ln n$, therefore either (for $np\le 40$)
	$$
	\djeden\le 40.1\dO =O(\ln n)$$  or (for $np > 40$) 
	$$\djeden=\frac{n}{m}(mp)^2\le n^{\eps} (1.2\dO)^2 = O(n^{\eps}\ln^2n).$$
	If there exist two vertices  $v,v'\in\SMALLv$, which are at distance $t$, $t\le 4$, in $\G$, then there exists a path $vw_1v_1w_2\ldots v_{t-1}w_tv'$ of length $2t$ in  $\B$ with which $\G$ is associated and in intersection graph $\G$ the set $\W'(\{v,v'\})\setminus \{w_1,\ldots,w_t\}$ is of cardinality at most $0.2\dO$.
	The number of possible paths of the form  $vw_1v_1w_2\ldots v_{t-1}w_tv'$ is upper bounded by $n^{t+1}m^{t}$. Probability that the path $vw_1v_1w_2\ldots v_{t-1}w_tv'$ is present in $\Bnmp$ is $p^{2t}$. 
	Moreover, under assumption that the path  $vw_1v_1w_2\ldots v_{t-1}w_tv'$ is present in $\Bnmp$, the cardinality of the set $\W'(\{v,v'\})\setminus \{w_1,\ldots,w_t\}$ has the binomial distribution $\Bin{m-t}{2(\dO/m) - p^2}$ with the expected value $\mu=(m-t)(2(\dO/m) - p^2)=(1-o(1))2\dO$ as 
	$
	d_0/m\ge p (1-e^{-(n-1)p})\ge  p (1-e^{-1}) \min\{1,(n-1)p\}\gg p^2
	$ for $p\le \sqrt{3\ln n/m} = o(1)$. Since $0.2\dO=0.1\cdot(1+o(1))\mu$, 
	by Chernoff's inequality the probability that there exist two SMALL vertices at distance at most $4$ is upper bounded by
	\begin{multline*}
	\sum_{t=1}^{4}n^{t+1}m^{t}p^{2t}\exp(-\psi(0.1(1+o(1)))(1+o(1))2\dO)\\
	=O(1)n \djeden^4 n^{-4/3}=
	O(n^{-1/3}n^{4\eps}\ln^8n)
	=o(1).
	\end{multline*}

	Now we will prove property {\bf P5}  as it will be needed in the proof of {\bf P3}.
	
	\noindent {\bf P5}
	Let $k=\frac{\ln n}{\ln\ln n}\max\{4,np\}$. Recall that $V(w)$ has the binomial distribution $\Bin{n}{p}$ and that we have  $p\le \sqrt{3\ln n/m}$. Therefore
	\begin{align*}
	\Pra{\exists_{w\in \W} V(w)\ge k}
	&\le m\binom{n}{k}p^{k}\\
	&\le m \left(\frac{enp}{k}\right)^{k-2} (np)^2\\
	&\le m 
	\exp
	\left((k-2)
	\left(
	1+\ln (np) - \ln \frac{\ln n}{\ln \ln n} - \ln \max\{4,np\}
	\right)
	\right)\\
	&\quad\cdot n^2\frac{3\ln n}{m}\\
	&\le  n^{-3}\cdot n^2\cdot 3\ln n=o(1).
	\end{align*}

	\noindent{\bf P3} The structure of $\mathcal{G}\left(n,m,p\right)$ differs a lot depending on the value $np$. Therefore we will need to divide the proof into cases:
	
	\begin{align}
	\label{case1}
	np&\le 40\quad\text{and}\quad |\Set|\le \frac{n}{\djeden^4};\\
	\label{case2}
	np&\le 40\quad\text{and}\quad \frac{n}{\djeden^4}\le |\Set|\le \frac{n}{\djeden};\\
	\label{case3}
	np&> 40.
	\end{align}

	First assume that \eqref{case1} is fulfilled. We will first prove that in this case {\whp} for any $\K\subseteq \V$, if $\K\le n/\djeden^3$ we have $W''(\K)\le 2|\K|$.
	
	\begin{align*}
	\Pra{\exists_{\K\subseteq \V, |\K|\le n/\djeden^3}W''(\K)\ge 2|\K|}
	&\le 
	\sum_{k=2}^{n/\djeden^3}\binom{n}{k}\binom{m}{2k}\left[\binom{k}{2}p^2\right]^{2k}\\
	&\le
	\sum_{k=2}^{n/\djeden^3}
	\left[
	\frac{en}{k}\cdot\frac{e^2m^2}{2^2k^2}\cdot\frac{k^4}{4}p^4
	\right]^k\\
	&\le 
	\sum_{k=2}^{n/\djeden^3}
	\left[
	\frac{e^3}{16}\cdot\frac{k}{n}(nmp^2)^2 
	\right]^k\\
	&\le 
	\sum_{k=2}^{n/\djeden^3}
	\left(
	\frac{e^3}{16} \cdot\frac{1}{\djeden}
	\right)^k
	=o(1).
	\end{align*}
	Moreover if $np\le 40$, and {\bf P5} is fulfilled, then 
	\begin{equation}\label{RownanieVw}
	\forall_{w\in\W}V(w)\le\frac{\ln n}{\ln\ln n}\max\{4,np\}\le \frac{40\ln n}{\ln\ln n}.
	\end{equation}
	
	Assume that $\G$ is an intersection graph such that \eqref{RownanieVw} is true and 
	\begin{equation}\label{RownanieK}
	\forall_{\K\subseteq \V,|\K|\le n/d^3} 
	W''(\K)\le 2|\K|,
	\end{equation}
	are fulfilled. In $\G$, let $\Set\subseteq \LARGEv$ be such that $|\Set|\le n/\djeden^4$. We will find a lower bound for $W'(\Set)$. Let $\B$ be the bipartite graph associated with $\G$. Then $W'(\Set)$ is at least the number of edges in $\B$ between $\Set$ and $\W'(\Set)\setminus\W''(\Set)$. There are $\sum_{v\in \Set}W'(v)$ edges between $\Set$ and $\W'(\Set)$ in $\B$ and at most $\sum_{w\in \W''(\Set)}V(w)$ of them are incident to some features from $\W''(\Set)$, therefore by \eqref{RownanieVw} and \eqref{RownanieK}

	\begin{align*}
	W'(\Set)&\ge \sum_{v\in \Set}W'(v)-\sum_{w\in \W''(\Set)}V(w)\\
	&\ge 
	0.1 \dO |\Set|- 2|\Set|\frac{40\ln n}{\ln \ln n}
	\ge 0.1 (1-o(1))\dO |\Set|.
	\end{align*}
	Set $\K=\Set\cup \N(\Set)$. 
	Assume that  $N(\Set)\le b_1 \djeden |\Set|$. Then, using \eqref{RownanieK},
	as $\W'(\Set)\subseteq \W''(\K)$,
	\begin{align*}
	W'(\Set)
	&\le W''(\K)
	\le 2|\K|\le 2(|\Set|+b_1 \djeden |\Set|)
	\\
	&\le 2(1+b_1\cdot 41\cdot \dO)|\Set|=82(1+o(1))b_1 \dO |\Set|.
	\end{align*}
	Combining the above inequalities we get
	$$
	0.1 (1-o(1))\dO |\Set|\le 82(1+o(1))b_1 \dO |\Set|. 
	$$ 
	Thus
	$$
	\frac{1}{1000}=b_1\ge (1+o(1))\frac{1}{820},
	$$
	a contradiction. Therefore in case \eqref{case1} we have {\whp}
	$$
	N(\Set)\ge b_1 \djeden |\Set|.
	$$
	
	Now assume that \eqref{case2} is fulfilled and $\Set$ is any subset of $\V$, not necessarily a subset of LARGE. Note that $W(\Set)$, $|\Set|=s$, has the binomial distribution $\Bin{m}{1-(1-p)^s}$. Therefore, for large $n$, $W(\Set)$ stochastically dominates a random variable $X_s$ with the binomial distribution $\Bin{m}{sp(1-sp)}$ (recall that $sp\le np/\djeden = 1/mp=o(1)$). Therefore
	
	\begin{equation}\label{RownanieWS1}
	\begin{split}
	\sum_{\Set, n/\djeden^4\le|\Set|\le n/\djeden}&\Pra{W(\Set)<0.05|\Set|mp}\\
	&\le
	\sum_{s=n/\djeden^4}^{n/\djeden}\binom{n}{s}\Pra{X_s<0.05smp}\\
	&\le
	\sum_{s=n/\djeden^4}^{n/\djeden}\left(e\djeden^4\right)^s\exp\left(-\psi\left(\frac{0.05}{1-sp}\right)smp(1-sp)\right)\\
	&\le
	\sum_{s=n/\djeden^4}^{n/\djeden}\exp\left(
	s\left(
	1+4\ln(mp)+4\ln(np)-0.8mp
	\right)
	\right)=o(1).
	\end{split}
	\end{equation}
	
	Now we will show a similar bound in the case where \eqref{case3} is fulfilled. In this case, if $\Set\subseteq \LARGEv$, $|\Set|=s$, then for all $v\in \Set$
	$$
	W(v)\ge W'(v)\ge 0.1\dO\ge 0.09 mp.
	$$
	Moreover, for any $v\in\V$, given $W(v)$, the set $\W(v)$ is uniformly distributed among all subsets of $\W$ of cardinality $W(v)$.
	Therefore, if for all $v\in\Set$ we have $W(v)\ge 0.09pm$ then the probability that $W(\Set)$ is at most $0.05smp$ is at most 
	$$
	\binom{m}{0.05smp}\left(\frac{0.05smp}{m}\right)^{0.09smp}
	=
	(e^{1.25}\cdot 0.05\cdot sp)^{0.04smp}.
	$$
	Recall that $sp\le np/\djeden = 1/mp =o(1)$. Therefore
	\begin{equation}
	\begin{split}\label{RownanieWS2}
	\sum_{\Set\subseteq \V, |\Set|\le n/\djeden}&\Pra{\{W(\Set)\le 0.05|\Set|mp\} \cap \{\Set\subseteq\LARGEv\}}\\
	\le& 
	\sum_{\Set\subseteq \V, |\Set|\le n/\djeden}\Pra{\{W(\Set)\le 0.05|\Set|mp\} \cap \{\forall_{v\in\Set} W(v)\ge 0.09mp \}}\\
	\le& 
	\sum_{s=1}^{n/\djeden}\binom{n}{s}\left(
	e^{1.25}\cdot 0.05\cdot sp\right)^{0.04smp}\\
	\le&
	\sum_{s=1}^{n/\djeden}\exp\left(
	s\left(
	\ln n + 0.04mp(1.25+\ln 0.05 + \ln sp)
	\right)
	\right)
	=o(1).
	\end{split}
	\end{equation}
	Now we will find a bound on  the probability that for some  $\Set\subseteq \LARGEv$ the size of its neighbourhood $N(\Set)$ is smaller than $b_1 |\Set| \djeden$ with $b_1=0.001$. In both cases, \eqref{case2} and \eqref{case3}, the proof is similar.
	
	For any $\Set\subseteq\V$, given the value $W(\Set)$, $N(\Set)$ has the binomial distribution with parameters $n-|\Set|$ and $1-(1-p)^{W(\Set)}$, {\it i.e.} $\Bin{n-|\Set|}{1-(1-p)^{W(\Set)}}$. If $W(\Set)\ge 0.05|\Set|mp$ then, for large $n$, $N(\Set)$ stochastically dominates a random variable with the binomial distribution  $\Bin{n-\frac{n}{\djeden}}{0.04|\Set|mp^2}$ (note that $|\Set|mp^2\le 1$ as $|\Set|\le n/\djeden$).
	Therefore
	\begin{align*}
	&\Pra{\{N(\Set)\le 0.001 |\Set| \djeden\}\cap\{\Set\subseteq \LARGEv\}}\\
	&\le
	\Pra{N(\Set)\le 0.001 |\Set| \djeden\,\Big|\, W(\Set)\ge 0.05|\Set|mp}+\Pra{\{W(\Set)< 0.05|\Set|mp\} \cap \{\Set\subseteq\LARGEv\}}\\
	&\le \exp\left(-\psi\left(\frac{0.025}{1-\djeden^{-1}}\right)(1-\djeden^{-1})0.04 |\Set|nmp^2\right)+\Pra{\{W(\Set)< 0.05|\Set|mp\} \cap \{\Set\subseteq\LARGEv\}}.
	\end{align*}
	For $|\Set|\ge n/\djeden^{4}$
	$$
	\psi\left(\frac{0.025}{1-\djeden^{-1}}\right)(1-\djeden^{-1})0.04 \djeden
	\ge 0.03 \djeden.
	$$
	For $np> 40$
	$$
	\psi\left(\frac{0.025}{1-\djeden^{-1}}\right)(1-\djeden^{-1})0.04 nmp^2>1.4 mp>1.4\cdot 0.9\ln n>1.2\ln n. 
	$$
	Thus finally in the case \eqref{case2}, using \eqref{RownanieWS1}, we get
	\begin{align*}
	&\Pra{\exists_{\Set\subseteq \V,n/\djeden^4 \le|\Set|\le n/\djeden}\{N(\Set)\le 0.001 |\Set| \djeden\}\cap\{\Set\subseteq\LARGEv\}}\\
	&\le
	\sum_{\Set\subseteq \V,n/\djeden^4 \le|\Set|\le n/\djeden}\Pra{\{N(\Set)\le 0.001 |\Set| \djeden\}\cap\{\Set\subseteq\LARGEv\}}\\
	&\le
	\sum_{\Set\subseteq \V,n/\djeden^4 \le|\Set|\le n/\djeden}\Pra{N(\Set)\le 0.001 |\Set| \djeden\,\Big|\, W(\Set)\ge 0.05|\Set|mp}\\
	&\quad+\sum_{\Set\subseteq \V,n/\djeden^4 \le|\Set|\le n/\djeden}\Pra{W(\Set) < 0.05|\Set|mp}\\
	&\le \sum_{s=n/\djeden^4}^{n/\djeden} \binom{n}{s}\exp\left(-0.3s\djeden\right) +o(1)\\
	&\le \sum_{s=n/\djeden^4}^{n/\djeden}\exp\left(s\left(1+\ln\frac{n}{s}-0.3\djeden\right)\right)+o(1)\\
	&\le \sum_{s=n/\djeden^4}^{n/\djeden}\exp\left(s\left(1+4\ln \djeden-0.3\djeden\right)\right)+o(1) =o(1)
	\end{align*}
	
	Similarly in case \eqref{case3}, using \eqref{RownanieWS2},
	\begin{align*}
	&\Pra{\exists_{\Set\subseteq \V,|\Set|\le n/\djeden}\{N(\Set)\le 0.001 |\Set| \djeden\}\cap\{\Set\subseteq\LARGEv\}}\\
	&\le \sum_{s=1}^{n/\djeden} \binom{n}{s}\exp\left(-1.2s\ln n\right) +o(1)\\
	&\le \sum_{s=1}^{n/\djeden}\exp\left(s\left(\ln n-1.2\ln n\right)\right)+o(1)=o(1)
	\end{align*}
	Finally we get that in all cases {\whp}
	$$
	\forall_{\Set\subseteq \LARGEv, |\Set|\le n/\djeden} N(\Set)\ge 0.001 |\Set|\djeden.
	$$

	\noindent {\bf P4}	 	We will show the statement only for $N(v)$. $W(v)$ and $W'(v)$ have the binomial distribution $\Bin{m}{p}$ and $\Bin{m}{\dO/m}$, resp. therefore the proofs of the remaining statements are similar but easier. Note that if $\djeden>n/12$ then the statement holds trivially. In the latter case, given $W(v)\le 4mp$, 
	$N(v)$ is stochastically dominated by a binomial random variable with the binomial distribution $\Bin{n}{4\djeden/n}$ 
	(as $1-(1-p)^{W(v)}\le 1-(1-p)^{4mp}\le 4mp^2$).
	Having in mind that $\djeden\ge \dO\ge  (1+o(1))\ln n$ and $mp\ge \dO\ge (1+o(1))\ln n$
	\begin{align*}
	\Pra{\exists_{v\in \V} N(v)\ge 12\djeden}
	&\le 
	n\left(
	\Pra{N(v)\ge 12\djeden\, |\, W(v)\le 4mp}
	+
	\Pra{W(v)\ge 4mp}
	\right)\\
	&\le 
	n\binom{n}{12\djeden}\left(\frac{4\djeden}{n}\right)^{12\djeden}
	+
	n\binom{m}{4mp}p^{4mp}\\
	&\le 
	n\left(\frac{e}{3}\right)^{12\djeden}+n\left(\frac{e}{4}\right)^{4mp}
	\\
	&\le
	\exp(\ln n + 12(1-\ln 3)\djeden)+\exp(\ln n + 4(1-\ln 4)mp)=o(1)
	\end{align*} 
\end{proof}

\section{Deletable sets}\label{SectionDeletable}

In this section we will mainly analyse an instance $\G$ of the random intersection graph. $\G$ will have properties listed in the previous section and will be such that HAM does not find a Hamilton cycle in $\G$. We will first create a random subgraph $\G_q$ of $\G$ by randomly deleting connections between vertices and attributes in the bipartite graph $\B$ associated with $\G$. We will bound the probability that execution of HAM on $\G_q$ proceeds the same steps as on $\G$ and $\G_q$ has good expansion properties. As a tool we will use a so called \emph{deletable set} of edges. Namely we will bound the probability that the edges from $E(\G)\setminus E(\G')$ form a deletable set.

In the analysis of HAM we will use some additional notation. In most cases it is consistent with \cite{HamiltonAlg1}. If HAM terminates unsuccessfully on stage $k$ then  
\medskip
\begin{itemize}
	\item[] $\Ham(G)$ is the set of paths $P^{(1)},\ldots,P^{(M)}$ and edges of the form $\{u_0,u_1\}$, where:
	\begin{itemize}
		\item $P_1=P^{(1)},\ldots,P^{(M)}=P_k$ is the sequence of paths constructed by HAM, where $P^{(i+1)}$ is obtained from $P^{(i)}$ by a simple extension, cycle extension, or rotation and
		\item $\{u_0,u_1\}$ are endpoints of the paths on which HAM executes a cycle extension;     	  
	\end{itemize}
	\item[] $\END{G}=$\{$v\in\V$: there exists  a path $Q_s$  on stage $k$ with $v$ as an endpoint and $\bar{\delta}(Q_s)=t$ for some $1\le t\le T$\};
	\item[]and for $x\in \END{G}$
	\item[] 
	$\END{G,x}=$\{$v\in\V$: there exists  a path $Q_s$  on stage $k$ with $v$ and $x$ as endpoints and $\bar{\delta}(Q_s)=t$ for some $1\le t\le 2T$\};
\end{itemize}

Let $\G$ be such that its vertex set is divided into two disjoint sets $\LARGEv$ and $\SMALLv$ and HAM with input value $d$ terminates unsuccessfully on $\G$. Given a constant $b_2>0$, we call $X\subseteq E(\G)$ \emph{deletable with } $b_2$ if
\begin{itemize}
	\item[{\bf D1}] no edge of $X$ is incident to a vertex from $\SMALLv$;
	\item[{\bf D2}] for any $v\in \LARGEv$ there are at most $b_2d$ edges from $X$ incident to $v$;
	\item[{\bf D3}] $X\cap \Ham(\G)=\emptyset$.   
\end{itemize}

From the description of the algorithm it follows that 
\begin{equation}
\label{RownanieHG}
|\mathcal{H}(\G)|\le n(2T+2),
\end{equation} 
where $T$ is defined as in the algorithm, {\it i.e.} for $d=\Omega(\ln n)$ we have $T = o(\ln n)$.

Following the lines of the proof of Lemma 3.2 from \cite{HamiltonAlg1} one may show the following lemma (see the proof in Appendix D).

\begin{lem}\label{LemENDG}
	Let $\{\G_n\}_{n=1,2,\ldots}$ be a sequence of graphs on $n$ vertices  such that for large $n$ HAM with input value $d$, $d=d(n)=\Omega(\ln n)$, terminates unsuccessfully in stage $k$ on $\G_n$. Moreover let $\G'_n\subseteq \G_n$,\linebreak  $b_3$ $(b_3>0)$ and $b_4$ $(0\le b_4 < \min\{b_3,1/3\})$ be constants and $\delta_0=\max\{1,b_4\min\{\delta(\G'_n),d\}\}$. If for large $n$
	\begin{itemize}
		\item[(i)] $\delta(\G'_n)\ge 2$;
		\item[(ii)] There exists a set $\V'\subseteq \V$ such that:
		\begin{itemize}
			\item for every $v$ there are at most $\delta_0$ vertices from $\V\setminus\V'$ at distance at most $2$ from~$v$;
			\item for all $\Set\subseteq \V'$ such that $|\Set|\le n/d$, $N_{\G'_n}(\Set)\ge b_3 d|\Set|$;   
		\end{itemize}
		\item[(iii)] $\Ham(\G_n)\subseteq E(\G'_n)$;
	\end{itemize}
	then HAM terminates unsuccessfully in stage $k$ on $\G'_n$ and for any constant $b_5<(b_3-b_4)/2$ and large~$n$ 
	\begin{align*}
	|\END{\G'}|&\ge b_5 n;\\
	|\END{\G',x}|&\ge b_5 n,\quad \text{for any }x\in \END{\G'}. 
	\end{align*}  	
\end{lem}

\begin{rem}\label{RemDeletable}
	Let $\LARGEv$ and $\SMALLv$ be defined as in Lemma~\ref{LemProperties}.
	Let $\G$ be an instance of $\mathcal{G}\left(n,m,p\right)$ with properties {\bf P0}--{\bf P5} and let $d=\djeden$. Moreover 	
	let $X$ be deletable with $b_2<b_1$ in $\G$, $\G'=(\V,E(\G)\setminus X)$, and let $\V'=\LARGEv$. Then {\bf P0}, {\bf P3}, {\bf D1}, and {\bf D2} imply $(i)$ for $n$ large. Moreover, {\bf D3} implies $(iii)$, and {\bf D1}, {\bf D2}, {\bf P2}, and {\bf P3} imply $(ii)$ with $b_4=0$ and $b_3<b_1-b_2$.
\end{rem}

Assume that we have an instance $\B$ of $\Bnmp$ and intersection graph $\G$ associated with $\B$. Denote by $\B_q$ a random subgraph of $\B$ obtained by deleting each edge of $\B$ independently with probability\linebreak $q=\lambda/n$, where $\lambda$ is a positive constant. Moreover let $\G_q$ be a random subgraph of $\G$ associated with $\B_q$. For any $\G$ and its subgraph $\G_q$ let $X_q$ be the set of egdes included in $\G$ but absent in $\G_q$. Consider a probability space in which we first pick $\G$ according to the probability distribution of $\mathcal{G}\left(n,m,p\right)$ and then create $\G_q$. The outcome of the experiment is a pair of graphs $(\G,\G_q)$. Denote by $\G_q(n,m,p)$ the random graph $\G_q$ constructed in this manner and by 
$\Xnmp$ 
the set $X_q$. If an edge was present in $\G$ ($\B$, resp.) and is absent in $\G_q$ ($\B_q$, resp) we will say that the edge \emph{was deleted in} $\G_q$ ($\B_q$, resp).  

\begin{lem}\label{LemDeletable}
	Let $\G$ be an intersection graph on $n$ vertices and let $\LARGEv$ and $\SMALLv$ be defined as in Lemma~\ref{LemProperties}. If $\G$ has properties {\bf P0}--{\bf P5} and HAM terminates  unsuccessfully on $\G$ then
	\begin{equation*}
	\Pra{X_q(n,m,p)\text{ is deletable with $b_2=0.5b_1$}\,\Big|\,\mathcal{G}\left(n,m,p\right)=\G} 
	\ge (1+o(1))\exp(-2\lambda(2T+2))
	\end{equation*} 
	where $o(1)$ is uniformly bounded over all possible choices of $\G$.
\end{lem}
\begin{proof}
	Recall that $X_q$ is deletable if it has properties {\bf D1}, {\bf D2}, and {\bf D3}.

	Set features in $\W$ in any order $\{w_1,\ldots,w_m\}$. For any edge $\{v,v'\}$ let $w_{vv'}$ be the smallest (in order of $\W$) feature in $\W''(\{v,v'\})$ (neighbour of both $v$ and $v'$ in $\B$). 
	
	For any $v\in\V$, $v'\in \N(v)$ and $w\in\W'(v)$ define events:
	\begin{itemize}
		\item[] $A_{vv'}$ -- event that neither $\{v,w_{vv'}\}$ nor $\{w_{vv'},v'\}$ was deleted in $\B_q$; 
		\item[] $B_{vv'}$ -- event that $\{v',w_{vv'}\}$ was deleted in $\B_q$;
		\item[] $C_{vw}$ -- event that $\{v,w\}$ was deleted in $\B_q$. 
	\end{itemize}	
	Consider events 
	\begin{align*}
	A&=\bigcap_{v\in \SMALLv, v'\in \N(v)} A_{vv'}
	\cap 
	\bigcap_{\{v,v'\}\in \Ham(\G)} A_{vv'} ,
	\\
	B&=
	\bigcup_{v\in\LARGEv}\ 
	\bigcup_{\Set\subseteq \N(v), |\Set|=0.1b_1\dO}\ 
	\bigcap_{v'\in \Set}B_{vv'},
	\\
	C&=\bigcup_{v\in\LARGEv}\ 
	\bigcup_{\R\subseteq \W'(v), |\R|= 0.1\frac{\ln\ln n}{\ln n}b_1\dO}\ 
	\bigcap_{w\in \R}C_{vw},	
	\end{align*}
	and event
	$$
	A\cap (B\cup C)^c=A\cap B^c\cap C^c,
	$$
	where, for any event $D$, by $D^c$ we denote the complement of event $D$. First we will prove that if event $A\cap B^c\cap C^c$ occurs then $X_q$ is deletable with $b_2=0.5b_1=0.0005$.  
	$\bigcap_{v\in \SMALLv,v'\in\N(v)} A_{vv'}$ implies {\bf D1} and $\bigcap_{\{v,v'\}\in \Ham(\G)} A_{vv'} $ implies {\bf D3}. We are left with showing that $B^c\cap C^c$ implies {\bf D2}.

	Let $v\in\LARGEv$. Note that if an edge between  $v$ and $\N(v)$ is deleted in $\G_q$ then there is a deleted edge in $\B_q$ between $\W'(v)$ and $\N(v)$  or between $v$ and $\W'(v)$.
	
	Let $\Set\subseteq \N(v)$ be the largest set such that, for all $v'\in \Set$, all edges between $v'$ and $\W'(v)$ were deleted in $\B_q$. Event $B^c$ implies that $|\Set|\le 0.1b_1\dO\le 0.1b_1\djeden$, therefore there are at most $0.1b_1\djeden$ edges incident to $v$ deleted in $\G_q$ due to deletions of edges  between $\W'(v)$ and $\N(v)$ in $\B_q$.
	
	Let $\R\subseteq \W'(v)$ be the largest set such that, for all $w\in \R$, $\{v,w\}$ was deleted in $\B_q$. Event $C^c$ implies that $|\R|\le 0.1\frac{\ln\ln n}{\ln n}b_1\dO$. Property {\bf P5} implies that 
	\begin{align*}
	V_{\G_q}(\R)
	&\le |\R|\frac{\ln n}{\ln\ln n}\max\{np,4\}\\
	&\le 0.1\frac{\ln\ln n}{\ln n}b_1\dO \frac{\ln n}{\ln\ln n}\max\{np,4\}\\
	&\le 0.4 b_1\dO \max\{np,1\}\le 0.4 b_1mp \min\{np,1\}\max\{np,1\}=0.4b_1\djeden. 
	\end{align*}
	Therefore there are at most $0.4b_1\djeden$ edges between $v$ and $\N(v)$ deleted in $\G_q$ due to deletions of edges between $v$ and $\W'(v)$ in $\B_q$.
	
	In conclusion, $B^c\cap C^c$ and {\bf P5} imply that for any vertex $v\in \LARGEv$ there are at most $0.5b_1\djeden$ edges incident to $v$ in $X_q$.
	
	Now we will bound the probability that, given $\G$ with properties {\bf P0}--{\bf P5}, 
	$
	A\cap (B\cup C)^c,
	$
	occurs. Let $\PraG{\cdot}=\Pra{\cdot\, |\, \mathcal{G}\left(n,m,p\right) = \G}$.
	
	In what follows we will use the fact that, given $\G=\mathcal{G}\left(n,m,p\right)$, the edges in $\B_q$ are deleted independently. Moreover, if we know that some set of edges of $\B$, say $E$, was not deleted in $\B_q$, then for any other set of edges $E'$ of $\B$ (possibly intersecting with $E$) the probability that all edges from $E'$ were not deleted (were deleted, resp.) is at least $(1-q)^{|E'|}$ (at most $q^{|E'|}$, resp.). 
	
	Now we will bound the number of edges incident to vertices from $\SMALLv$. Recall that by {\bf P1}, if $\dO\ge 2\ln n$ then $|\SMALLv|=0$. Therefore we need only to consider the case $\dO\le 2\ln n$. If $m\ge n^{1-\eps}$ and $d_0\le 2\ln n$ then $\djeden=O(n^{\eps}\ln^2 n)$ (see discussion in the proof of {\bf P2}). 
	Therefore properties {\bf P1} and {\bf P4} imply that there are at most 
	$$|\SMALLv|\cdot 12\djeden
	=O(n^{1/3}n^{1/25}\ln^2 n)=o(n)
	$$ 
	edges in $\G$ incident to  vertices from $\SMALLv$. In addition, by \eqref{RownanieHG}, $|\Ham(\G)|\le n(2T+2)$. Therefore    
	$$
	\PraG{A}\ge ((1-q)^2)^{o(n)}\cdot (1-q)^{2|\Ham(\G)|}=(1-o(1))\exp(-2\lambda(2T+2)).
	$$
	Moreover by {\bf P4} and the fact that $p\le \sqrt{3\ln n /m}=o(1)$
	\begin{align*}
	\PraG{B|A}&
	\le \sum_{v\in\LARGEv}\ \sum_{\Set\subseteq \N(v), |\Set|=0.1b_1\dO}\Pra{\bigcap_{v'\in \Set}B_{vv'}\,\Bigg|\, A}
	\\
	&\le n\binom{12\djeden}{0.1b_1\dO}q^{0.1b_1\dO}=n\left(\frac{12e}{0.1b_1}\cdot2\max\{np,1\}\cdot\frac{\lambda}{n}\right)^{0.1b_1\dO}\\
	&\le n\left(O(1)\max\{p,n^{-1}\}\right)^{0.1b_1(1+o(1))\ln n} =o(1). 
	\end{align*}
	and
	\begin{align*}
	\PraG{C|A}&\le
	\sum_{v\in \LARGEv}\, \sum_{\R\subseteq \W'(v), |\R|=0.1b_1\frac{\ln\ln n}{\ln n}\dO}\, \PraG{\bigcap_{w\in \R}C_{vw}\,\Bigg|\,A}
	\\
	&\le n\binom{4\dO}{0.1b_1\frac{\ln\ln n}{\ln n}\dO}q^{0.1b_1\frac{\ln\ln n}{\ln n}\dO}\\
	&\le n\left(\frac{e4}{0.1b_1}\cdot\frac{\ln n}{\ln \ln n}\cdot\frac{\lambda}{n}\right)^{0.1b_1\frac{\ln\ln n}{\ln n}\dO}=o(1).	
	\end{align*}
	
	Therefore, given an instance $\G$ of $\mathcal{G}\left(n,m,p\right)$ with properties {\bf P0}--{\bf P5} we have
	\begin{multline*}
	\PraG{A\cap (B\cup C)^c}=\PraG{A}\left(1-\PraG{B\cup C|A}\right)\\
	\ge \PraG{A}\left(1-\PraG{B|A}-\PraG{C|A}\right)\ge (1-o(1))\exp(-2\lambda(2T+2)), 
	\end{multline*}
	where $o(1)$ is uniformly bounded over all possible choices of $\G$ with properties {\bf P0}--{\bf P5}.
\end{proof}

\section{Probability of HAM terminating unsuccessfully}\label{SectionProof}

Recall that we consider pairs $(\G,\G_q)$, where $\G$ is chosen according to the probability distribution of $\mathcal{G}\left(n,m,p\right)$ and $\G_q$ is a random subgraph of $\G$. In previous sections we already bounded the probability that in the case when HAM terminates unsuccessfully on $\G$  during the execution of HAM on  $\G_q$ a linear number of longest paths is created. In this section we will additionally prove that it is highly unlikely that among the large number of paths in $\G_q$ none of them close a cycle in $\G$. This will lead us to the bound on the probability that HAM terminates unsuccessfully  on $\mathcal{G}\left(n,m,p\right)$. 

Define $\LARGEv$ and $\SMALLv$ as in Lemma~\ref{LemProperties}. For a pair $(\mathcal{G}\left(n,m,p\right),\G_q(n,m,p))$ analysed in the previous section define events
\begin{itemize}
	\item[] $D_1$ -- $\mathcal{G}\left(n,m,p\right)$ has properties {\bf P1}--{\bf P5} with $b_1=1/1000$;
	\item[] $D_2$ -- $\delta(\mathcal{G}\left(n,m,p\right))\ge 2$ ($\mathcal{G}\left(n,m,p\right)$ has property {\bf P0}); 
	\item[] $D_3$ -- HAM terminates unsuccessfully on $\mathcal{G}\left(n,m,p\right)$; 
	\item[] $E_1$ -- HAM terminates unsuccessfully on  $\G_q(n,m,p)$; 
\end{itemize}
We consider also two other events contained in $E_1$
\begin{itemize}
	\item[] $E_2$ -- there is no edge $\{v,v'\}\in \Xnmp$ such that $v\in \END{\G_q(n,m,p)}$\\ and $v'\in \END{\G_q(n,m,p),v}$;
	\item[] $E_3$ --  $|\END{\G_q(n,m,p)}|\ge n/5000$ and  $|\END{\G_q(n,m,p),v}|\ge n/5000$,\\ for all $v\in \END{\G_q(n,m,p)}$.
\end{itemize}
Let moreover
$$
D=D_1\cap D_2\cap D_3\quad\text{and}\quad E=E_1\cap E_2\cap E_3.
$$

Note that by Lemma~\ref{LemENDG} and Remark~\ref{RemDeletable}, $D$ and the event that $X_q(n,m,p)$ is deletable with constant $b_2=0.5b_1=1/2000$ imply event $E$. 
Therefore by Lemma~\ref{LemDeletable}
\begin{equation}\label{RownanieDeletable}
\begin{split}	
\Pra{D\cap E}&\ge 
\Pra{E|D}\Pra{D}\\
&\ge \Pra{X_q(n,m,p)\text{ is deletable with $b_2=0.5b_1$}|D}\Pra{D}\\
&\ge (1-o(1))\exp(-2\lambda(2T+2))\Pra{D}.
\end{split} 
\end{equation}

Now we will find an upper bound on the probability $\Pra{D\cap E}$.
Note that $\G_q(n,m,p)$ is distributed as $\G(n,m,p')$ associated with $\B(n,m,p')$, where $p'=p(1-q)$.

First we bound the probability of event $F^c$, where
\begin{itemize}
	\item[] $F$ - $\forall_{\Set,|\Set|=n/5000}$ $W(\Set)=W_{\G(n,m,p')}(\Set)\ge 0.3 m\min\left\{\frac{np'}{5000},1\right\}$ in $\G(n,m,p')$.
\end{itemize}
Assume that $n\le m$. 
Recall that if $|\Set|=n/5000$ then  $W(\Set)$ has the  binomial distribution\linebreak $\Bin{m}{1-(1-p')^{n/5000}}$ and $1-(1-p')^{n/5000}\ge 0.5\min\{np'/5000,1\}$. As $mp'\ge (1+o(1))\ln n$ and $n\le m$
$$
\E W(\Set) \ge 0.5 m  \min\left\{\frac{np'}{5000},1\right\}\ge 
\frac{n}{5000}0.5 \min\left\{mp',\frac{5000m}{n}\right\}\ge 
\frac{n}{5000}\cdot 2500.
$$
Therefore, if $n\le m$ then
\begin{equation}\label{RownanieE}
\begin{split}
\Pra{F^c}&=\Pra{\exists_{\Set\subseteq\W, |\Set|=\frac{n}{5000}}W(\Set)\le 0.3 m\min\left\{\frac{np'}{5000},1\right\}}\\
&\le 
\binom{n}{n/5000}\exp\left(-\psi\left(0.6\right)0.5m\min\left\{\frac{np'}{5000},1\right\}\right)\\
&\le
\exp\left(
\frac{n}{5000}
\left(
1+\ln 5000 -\psi(0.6)\cdot 2500
\right)
\right)
\\
&\le
\exp
\left(
\frac{n}{5000}
\left(
1+\ln 5000 - 233
\right)
\right)
\le \exp(-0.04n).
\end{split}
\end{equation}  

Now assume that $n\ge m$. Then $np'= (n/m)mp'\ge (1+o(1))\ln n$, therefore, for $n$ large, we have $0.3 m\min\left\{\frac{np'}{5000},1\right\}=0.3m$. Let 
$$
V''=|\{v\in\V: W(v)\le 0.1 mp'\}|.
$$ 
Note that $\Pra{W(v)\le 0.1 mp'}\le\exp(-\psi(0.1)mp')\le n^{-2/3}$ and $W(v)$, $v\in\V$, are independent. $V''$ is therefore stochastically dominated by a random variable with the binomial distribution $\Bin{n}{n^{-2/3}}$. In addition, given $W(v)=t\ge 0.1mp'$, the set $\W(v)$ is distributed uniformly over all subsets of $\W$ of cardinality $t$ (independently of all other vertices from $\V$). Thus, for $n\ge m$,   
\begin{align*}
\Pra{F^c}&\le 
\Pra{F^c|V''\le n^{1/2}}+\Pra{V''\ge n^{1/2}}\\
&\le 
\binom{n}{\frac{n}{5000}}
\binom{m}{0.3m}
\left(
\frac{0.3m}{m}
\right)^{\left(\frac{n}{5000}-n^{1/2}\right)\cdot 0.1mp'}
+
\binom{n}{n^{1/2}}(n^{-2/3})^{n^{1/2}}\\
&\le
(e5000)^{\frac{n}{5000}}\left(\frac{e}{0.3}\right)^{0.3m}0.3^{(1+o(1))\frac{n}{5000}mp'}
+
\left(en^{1/2}n^{-2/3}\right)^{n^{1/2}}\\
&\le \exp(-n^{1/2}).
\end{align*}

Recall that $\B_q(n,m,p)$ is distributed as $\B(n,m,p')$. Therefore $\B(n,m,p)$ is distributed as a random bipartite graph $\B_q(n,m,p)\cup\B(n,m,q')$, where $\B_q(n,m,p)$ and $\B(n,m,q')$ are independent and we have
$q'=pq/(1-p+pq)=(1+o(1)) pq$ (recall that we assume that $\ln n = o(m)$ {\it i.e.} $p\le \sqrt{3\ln n/m}=o(1)$).  
If $E_2$ occurs then in $\B(n,m,q')$ for all $v\in \END{\G'}$ there is no edge between $\W(\END{\G',v})$ and $v$, therefore  
\begin{align*}
\Pra{E_2|E_1\cap E_3\cap F}
&\le
\prod_{v\in\END{\G'}}(1-q')^{W(\END{\G',v})}
\\
&\le 
\exp\left(
-\frac{n}{5000}\cdot(1+o(1))p\cdot\frac{\lambda}{n}\cdot0.3m\min\left\{\frac{np'}{5000},1\right\}
\right)\\
&\le 
\exp\left(
-(1+o(1))\lambda\cdot\frac{ 3}{50000}\cdot\min\left\{\frac{\djeden}{5000},mp\right\}
\right)\\
&\le \exp(-\lambda'\ln n)
\end{align*}
for $\lambda'=\lambda\cdot 10^{-8}$.

Therefore by \eqref{RownanieE} and the above calculation
\begin{equation*}
\Pra{D\cap E}\le 
\Pra{E_2|E_1\cap E_3\cap F} + \Pra{F^c}\le (1+o(1))
\exp\{-\lambda'\ln n\}.
\end{equation*}
Combining this equation with \eqref{RownanieDeletable} we get 
$$
(1-o(1))\exp(-2\lambda(2T+2))\Pra{D}\le (1+o(1))\exp\{-\lambda'\ln n\},
$$
{\it i.e.}
$$
\Pra{D}\le (1+o(1))\exp\{-\lambda'\ln n+2\lambda(2T+2)\}=o(1).
$$
Therefore by the above equation and Lemma~\ref{LemProperties}

\begin{multline*}
\Pra{\{\delta(\mathcal{G}\left(n,m,p\right))\ge 2\}\cap\{\text{HAM finishes unsuccessfully on }\mathcal{G}\left(n,m,p\right)\}}\\
=\Pra{D_2\cap D_3}\le \Pra{D_1\cap D_2\cap D_3} + \Pra{D_1^c}=o(1).
\end{multline*}

\section{Proof for $m\le n^{1-\eps}$}\label{SectionSmallm}

We will not give the whole proof in the case $m\le n^{1-\eps}$ ($\eps=1/25$) and $mp^2\le 1$ as, in its main part, it follows the lines of the proof for $m\ge n^{1-\eps}$. We just note several adjustments that we make. 

We leave the definition of $d$ but we redefine $\SMALLv$ and $\LARGEv$.
$$
\SMALLv_*=\{v\in \V: W'(v)\le 6\cdot 10^{-3}mp\},\quad \LARGEv_*=\{v\in \V: W'(v) > 6\cdot 10^{-3}mp\}.
$$

We need to add that in this case 

\begin{equation}
\label{RownanienpMaleM}
np= (n/m)mp\ge (1+o(1))n^{\eps}\ln n\quad\text{and}\quad mp=(1+o(1))\dO.
\end{equation}

We redefine the properties.
\begin{itemize}
	\item[{\bf P0$_*$}] $\delta(\G)\ge np/2\ge  n^{\eps}\ln n/2$;
	\item[{\bf P1}$_*$]  $|\SMALLv_*|\le n^{\eps}$ and \\
	if $mp\ge 2\ln n$ then $|\SMALLv_*|=\emptyset$;
	\item[{\bf P3}$_*$] 		
	For all $\Set\subseteq \LARGEv_*$, if $|\Set|\le n/\djeden$ then $N(\Set)\ge b_1 \djeden |\Set|$. 
	\item [{\bf P4\ }] 		
	For all $v\in\V$ 
	$$W(v)\le 4mp,\quad W'(v)\le 4\dO,\quad N(v)\le 12\djeden.$$
	\item [{\bf P5\ }] For all $w\in\W$ 
	$$V(w)\le \frac{\ln n}{\ln\ln n} \max\{np,4\}.$$
\end{itemize}

Now we may prove the following analogue of Lemma~\ref{LemProperties}

\begin{lem}\label{LemProperties2}
	Let $\mathcal{G}\left(n,m,p\right)$ be a random intersection graph and
	$$
	\dO\ge (1+o(1))\ln n.
	$$
	Moreover let 
	$$
	\SMALLv_*=\{v\in \V: W'(v)\le 6\cdot 10^{-3}mp\},\quad \LARGEv_*=\{v\in \V: W'(v) > 6\cdot 10^{-3}mp\}.
	$$
	If $m < n^{1-\eps}$ $(\eps=1/25)$ then {\whp} $\mathcal{G}\left(n,m,p\right)$ has properties  {\bf P1}$_{*}$, {\bf P3}$_{*}$ with $b_1=0.001$, {\bf P4}, and {\bf P5}.
\end{lem}

\begin{rem}\label{RemDelta2lnn}
	If $m < n^{1-\eps}$ then
	\begin{equation}
	\label{RownanieDeltaDuze}
	\Pra{\{\delta(\mathcal{G}\left(n,m,p\right))\ge 2\}\cap\{\delta(\mathcal{G}\left(n,m,p\right))\ge np/2\}^c}=o(1).
	\end{equation}
\end{rem}

The proofs of Lemma~\ref{LemProperties2} and Remark~\ref{RemDelta2lnn} are presented in Appendix~C. 

We also redefine the notion of a deletable set. Let $\G$ be such that HAM terminates unsuccessfully on $\G$. Given constant $b_2>0$, we call $X\subseteq E(\G)$ \emph{deletable with } $b_2$ if
\begin{itemize}
	\item[{\bf D1}$_*$] for any $v\in \SMALLv_*$ there are at most $b_2\dO$ edges from $X$ incident to $v$;
	\item[{\bf D2}$_*$] for any $v\in \LARGEv_*$ there are at most $b_2\djeden$ edges from $X$ incident to $v$;
	\item[{\bf D3\ }] $X\cap \Ham(G)=\emptyset$.   
\end{itemize}

Now we will show that with the above definition we may give an analogue of Remark~\ref{RemDeletable}. 

\begin{lem}
	Let $\G$ be an instance of $\mathcal{G}\left(n,m,p\right)$ and $\LARGEv_*$ and $\SMALLv_*$ be defined as in Lemma~\ref{LemProperties2}. If $\G$ has properties  {\bf P0}$_{*}$, {\bf P1}$_{*}$, {\bf P3}$_{*}$, {\bf P4}, {\bf P5} and   $X$ is deletable in $\G$ with a constant $b_2<b_1$. Then for $\G'=(\V,E(\G)\setminus X)$ and large $n$ the assumptions of Lemma~\ref{LemENDG} with $\V'=\LARGEv_*$ are fulfilled with any constants $0<b_3<b_1-b_2$ and $0<b_4<\min\{b_3,1/3\}$.	
\end{lem}

\begin{proof}

	Recall that $mp=o(np)$, $\dO=(1+o(1))mp$, and $np=o(\djeden)$. Therefore
	By {\bf P0}$_*$, {\bf P3}$_*$, {\bf D1}$_*$, and {\bf D2}$_*$ 
	$$
	\delta(\G')\ge 
	\min\{\delta(\G)-b_2\dO,(b_1-b_2)\djeden\}\ge (1+o(1))np/2\ge (1+o(1))n^{\eps}\ln n/2. 
	$$
	This implies (i).
	
	Now we show the first part of  (ii).
	By the above calculation we have
	$$\min\{\delta(\G'),\djeden\}\ge (1+o(1))n^{\eps}\ln n/2,$$ therefore by {\bf P0}$_*$ and {\bf P1}$_*$ 
	$$
	|\SMALLv_*|\le n^\eps = o(\min\{\delta(\G'),\djeden\})< b_4\min\{\delta(\G'),\djeden\}
	$$ 
	for any constant $b_4$ and large $n$.
	This implies the first part of  $(ii)$.
	The second part of (ii) follows by  {\bf P3}$_*$ and {\bf D2}$_*$. 
	Finally $(iii)$ follows by {\bf D3}.
\end{proof}

Now we will show an analogue of Lemma~\ref{LemDeletable}.

\begin{lem}\label{LemDeletable2}
	Let $\LARGEv_*$ and $\SMALLv_*$ be defined as in Lemma~\ref{LemProperties2} and $\G$ be an intersection graph on $n$ vertices with properties  {\bf P0}$_{*}$, {\bf P1}$_{*}$, {\bf P3}$_{*}$, {\bf P4}, {\bf P5} such that HAM finishes unsuccessfully on $\G$. Then
	\begin{equation*}
	\Pra{X_q(n,m,p)\text{ is deletable with $b_2=0.5b_1$}\,\Big|\,\mathcal{G}\left(n,m,p\right)=\G} 
	\ge (1+o(1))\exp(-2\lambda(2T+2))
	\end{equation*} 
	where $o(1)$ is uniformly bounded over all possible choices of $\G$.
\end{lem}
\begin{proof}
	We redefine events $A$ and $B$. For any $v\in \V$ and $w\in\W$ let 
	$A^*_{vw}$ be event that $\{v,w\}$ was not deleted in $\B_q$.
	Moreover
	\begin{align*}
	A_*&=\bigcap_{v\in \SMALLv_*, w\in\W(v) } A^*_{vw}
	\cap 
	\bigcap_{\{v,v'\}\in \Ham(\G)} A_{vv'} ,\\
	B_*&=
	\bigcup_{v\in\V}\ 
	\bigcup_{\Set\subseteq \N(v), |\Set|=0.1b_1\dO}\ 
	\bigcap_{v'\in \Set}B_{vv'},\\
	\\
	C&=\bigcup_{v\in\LARGEv_*}\ 
	\bigcup_{\R\subseteq \W'(v), |\R|= 0.1\frac{\ln\ln n}{\ln n}b_1\dO}\ 
	\bigcap_{w\in \R}C_{vw}.\end{align*}

	Recall that for  $v\in\SMALLv_*$, if $mp\le 2\ln n$ then we have $W(v)\le 6\cdot 10^{-3} mp\le 12\cdot 10^{-3}\ln n$. Otherwise by {\bf P1}$_*$ $\SMALLv_*=\emptyset$. Therefore by {\bf P1}$_*$ and \eqref{RownanieHG}  
	\begin{align*}
	\PraG{A_*}\ge (1-q)^{n^{\eps}\cdot 12\cdot 10^{-3}\ln n}(1-q)^{2|\Ham(\G)|}
	\ge (1+o(1))\exp(-2\lambda(2T+2)).
	\end{align*} 
	Moreover, similar bounds as for $\PraG{B|A}$ and $\PraG{C|A}$ in the case $m\ge n^{1-\eps}$ give 
	$$
	\PraG{B_*|A_*}=o(1)\quad\text{and}\quad\PraG{C|A_*}=o(1).
	$$ 
	Therefore
	$$
	\PraG{A_*\cap(B_*\cup C)^c}\ge (1+o(1))\exp(-2\lambda(2T+2)).
	$$
	We are left with showing that $A_*\cap(B_*\cup C)^c$ implies that $X_q$ is deletable. The same arguments as before give that $A_*$ implies {\bf D3} and $(B_*\cup C)^c$ implies {\bf D2}$_*$.  Moreover $A_*$ with $B_*$ imply {\bf D1}$_*$.

In the remaining part of the proof we only replace
$D_1$ and $D_2$ by
\begin{itemize}
	\item[] $D_{1*}$ -- $\mathcal{G}\left(n,m,p\right)$ has properties  {\bf P1}$_{*}$, {\bf P3}$_{*}$, {\bf P4}, {\bf P5} with $b_1=1/1000$;
	\item[] $D_{2*}$ -- $\delta(\mathcal{G}\left(n,m,p\right))\ge np/2$ ($\mathcal{G}\left(n,m,p\right)$ has property {\bf P0}$_*$); 
\end{itemize}
and at the end, using \eqref{RownanieDeltaDuze}, we get  
\begin{align*}
&\Pra{\{\delta(\mathcal{G}\left(n,m,p\right))\ge 2\}\cap\{\text{HAM finishes unsuccessfully on }\mathcal{G}\left(n,m,p\right)\}}\\
&\le \Pra{\{\delta(\mathcal{G}\left(n,m,p\right))\ge np/2\}\cap\{\text{HAM finishes unsuccessfully on }\mathcal{G}\left(n,m,p\right)\}}\\
&\quad+\Pra{\{\delta(\mathcal{G}\left(n,m,p\right))\ge 2\}\cap \{\delta(\mathcal{G}\left(n,m,p\right))\ge np/2\}^c}\\
&=\Pra{D_{2*}\cap D_3}+o(1)\\
&\le \Pra{D_{1*}\cap D_{2*}\cap D_3} + \Pra{D_{1*}^c}+o(1)=o(1).
\end{align*}

\end{proof}

\section*{Appendix A}

First consider the case $mp^2\le 1$.
In this case by Lemma~\ref{LemProperties} (property {\bf P4}) {\whp} each vertex in $\mathcal{G}\left(n,m,p\right)$ has degree at most $12d$. Therefore in stage $k$, $1\le k\le n$, each of the considered paths has at most $24d$ rotations (at most $12d$ rotations for each of its ends). As we consider only those paths which result from at most $2T+1$ rotations, the total number of rotations made during the execution of stage $k$ is at most $(12d)^{2T+2}=O(n^{2+\eps})$, for any constant $\eps>0$. Summing over all stages we get that {\whp} the algorithm makes $O(n^{3+\eps})$ rotations. Considering the time needed to execute a rotation we get that {\whp} HAM terminates in $O(n^{4+\eps})$ time on $\mathcal{G}\left(n,m,p\right)$. In the case $mp^2\ge 1$ we have $T \le 2$ therefore in each stage the number of rotations is at most  $O(n^2)$. Thus the algorithm terminates in $O(n^4)$ rounds.       

\section*{Appendix B}

\begin{proof}[ of Theorem~\ref{LemmaDegree}]
	
	The following result is a special case of 
	Lemma 5.1 from \cite{RIGcoupling2}.

	\begin{lem}\label{LemmaDegree1}
		If $\ln^2n=o(m)$ and 
		\begin{equation*}
		p(1-(1-p)^{n-1})=
		\frac{\ln n + c_n}{m},
		\end{equation*}
		then
		$$
		\lim_{n\to\infty}\Pra{\delta(\mathcal{G}\left(n,m,p\right))\ge 1}=
		\begin{cases}
		0&\text{ for }c_n\to -\infty;\\
		e^{-e^{-c}}&\text{ for }c_n\to c\in (-\infty,\infty);\\
		1&\text{ for }c_n\to \infty.
		\end{cases}
		$$
	\end{lem}	
	Now we will find the probability of existence of a vertex of degree 1.  Recall that 
	$$p(1-(1-p)^{n-1})=\frac{\ln n + \ln a_n + c_n}{m},$$
	and
	$$
	a_n=
	\begin{cases}
	1&\text{for } m < (1-\eps)  n\ln n/\ln\ln n;\\
	\frac{np\ln n}{e^{np}-1}&\text{for } m > (1+\eps) n\ln n/\ln \ln n. 
	\end{cases}
	$$

	First consider the case  $m < (1-\eps)  n\ln n/\ln\ln n$. In this case $p(1-(1-p)^{n-1})=\frac{\ln n + c_n}{m}$ as $a_n=1$. We will show that if $c_n\to c\in (-\infty,\infty]$ then {\whp} there is no vertex of degree $1$ in $\mathcal{G}\left(n,m,p\right)$.  Note that in this case
	$$np \ge np(1-(1-p)^{n-1})=\frac{n(\ln n + c_n)}{m}>(1+o(1))\frac{\ln\ln n}{1-\eps}.$$ 
	Therefore, for large $n$, $np > (1+o(1))\ln\ln n/(1-\eps)> (1+\eps)\ln\ln n$, $\ln np = o(np)$, and $mp = (1+o(1)) \ln n$.

	Note that $v\in\V$ has degree $1$ in an intersection graph $\G$ associated with a bipartite graph $\B$  if there exist a feature $w$ and  vertex $v'$ such that $\V(w)=\{v,v'\}$ and for any other feature $w'\in\W\setminus\{w\}$ pair $\{v,w'\}$ is not an edge in $\B$, or $w'$ is adjacent only to $v$ in $\B$, or $\V(w')=\{v,v'\}$. Therefore, 
	for any vertex $v\in \V$ and $c_n\to c \in (-\infty,\infty]$,
	\begin{align*}
	\Pra{\deg(v)=1}
	&
	\le 
	\sum_{w\in\W}\sum_{v'\in \V}
	p^2(1-p)^{n-2}((1-p)+p(1-p)^{n-1}+p^2)^{(m-1)}\\
	&\le mp  \cdot   np \cdot \exp(-np+p) \\
	&\quad\cdot\exp\left(
	- mp (1-(1-p)^{n-1}) + O(mp^2+p)
	\right)\\
	&\le \exp(\ln mp + \ln np - np - \ln n - c_n + o(1)) \\
	& \le \exp(\ln ((1+o(1))\ln n)   - (1+o(1)) np - \ln n - c_n + o(1)) \\
	&\le \exp(\ln\ln n - (1+\eps)\ln \ln n -\ln n - c_n + o(1))\\
	&
	= o(n^{-1})
	\end{align*}
	Thus {\whp} there is no vertex of degree $1$ in $\mathcal{G}\left(n,m,p\right)$. This combined with Lemma~\ref{LemmaDegree1} completes the proof in this case.
	
	Now assume that $m>(1+\eps)n\ln n/\ln \ln n$. We will use the method of moments to get a Poisson approximation of the number of vertices of degree 1 (see for example Corollary 6.8 in \cite{KsiazkaJLR}). In this case $np<(mp\ln\ln n)/((1+\eps)\ln n)$ and, as a consequence, $a_n=(np\ln n)/(e^{np}-1)\to\infty$ (note that $x/(e^x-1)$ is decreasing for $x\in (0,\infty)$). Recall that $mp^2=o(1)$, {\it i.e.} also $np^2=o(1)$.
	Let $k$ be a constant integer and $\{v_1,\ldots,v_k\}\subseteq \V$. Note that $(m)_k$ is the number of choices of $w_1,\ldots,w_k$, which might be chosen by vertices $v_1,\ldots,v_k$, resp., $(n-k)p^2(1-p)^{n-2}$ is the probability that $w_i$, $1\le i\le k$, was chosen only by $v_i$ and one other vertex from $\V\setminus\{v_1,\ldots,v_k\}$, and $((1-p)^{k}+kp(1-p)^{n-1})^{(m-k)}$ is the probability that each feature from $\W\setminus \{w_1,\ldots,w_k\}$ either was not chosen by any vertex from $\{v_1,\ldots,v_k\}$ or was chosen by exactly one vertex from $\{v_1,\ldots,v_k\}$ and no other vertex. Therefore 
	\begin{align*}
	&\Pra{\deg(v_1)=\ldots=\deg(v_k)=1}\\
	&\ge (m)_k\left((n-k)p^2(1-p)^{n-2}\right)^k((1-p)^{k}+kp(1-p)^{n-1})^{(m-k)}\\
	&\ge (1+o(1)) (mp(1-(1-p)^{n-1}))^k\left(\frac{np(1-p)^{n-1}}{(1-(1-p)^{n-1})}\right)^k\\
	&\quad\cdot(1-kp+kp(1-p)^{n-1})^{(m-k)}
	\\
	&= (1+o(1)) \left(\frac{\ln n \cdot np}{(1-p)^{-(n-1)}-1}\right)^k\\ &\quad\cdot\exp(-kmp(1-(1-p)^{n-1})+O(mp^2+p))\\
	&= (1+o(1)) \left(\frac{\ln n \cdot np}{e^{np}-1}\right)^k\exp(-k(\ln n + \ln a_n + c_n))\\
	&= (1+o(1))  \left(\frac{e^{-c_n}}{n}\right)^{k}. 
	\end{align*}
	
	Now we will find an upper bound on $\Pra{\deg(v_1)=\ldots=\deg(v_k)=1}$. Below we set $i$, $0\le i\le k/2$, to be the number of edges induced by  $\{v_1,\ldots,v_k\}$ in $\mathcal{G}\left(n,m,p\right)$. We note that the probability that for a given feature $w$ set  $\V(w)$ is the same as some edge incident to a vertex from $\{v_1,\ldots,v_k\}$ is upper bounded by $kp^2$. Therefore
	\begin{align*}
	&\Pra{\deg(v_1)=\ldots=\deg(v_k)=1}\\
	&\le
	\sum_{i=0}^{k/2} (m)_{k-i}(k)_{2i}(np^2(1-p)^{n-2})^{(k-i)}\\
	&\quad\cdot ((1-p)^{k}+kp(1-p)^{n-1}+kp^2)^{(m-k)}\\
	&\le (1-kp+kp(1-p)^{n-1}+(kp)^2+kp^2)^{(m-k)}\\
	&\quad\cdot\sum_{i=0}^{k/2}k^{2i}(mp(1-(1-p)^{n-1}))^{k-i}
	\left(\frac{np(1-p)^{n-1}}{(1-(1-p)^{n-1})}\right)^{k-i}
	\\
	&=(1+o(1))\exp\left(-k(\ln n + \ln a_n + c_n)\right)
	\sum_{i=0}^{k/2} k^{2i}a_n^{k-2i}
	\\
	&= (1+o(1)) \left(\frac{e^{-c_n}}{n}\right)^k. 
	\end{align*}
	
	Therefore by the method of moments  the number of vertices of degree $1$ in $\mathcal{G}\left(n,m,p\right)$ tends to a random variable with the Poisson distribution $\Po{e^{-c}}$, where $c=\lim_{n\to\infty}c_n$.
	This combined with Lemma~\ref{LemmaDegree1} finishes the proof in the second case.
	
\end{proof}

\section*{Appendix C}

\begin{proof}[ of Remark~\ref{RemDelta2lnn}]
	In fact we will show that {\whp}
	\begin{equation}\label{RownanieVR}
	\forall_{\R\subseteq \W,|\R|\le p^{-1}}
	V(\R)\ge \frac{np|\R|}{2}+1.
	\end{equation}
	Note that \eqref{RownanieVR} implies that {\whp} each vertex is either isolated or of degree at least $np/2$. Therefore \eqref{RownanieDeltaDuze} follows by \eqref{RownanieVR}.
	
	We will now show \eqref{RownanieVR}. Recall that $np\ge (1+o(1))n^{\eps}\ln n$. 
	For $\R\subseteq \W$ with $|\R|=r\le p^{-1}$, $V(\R)$ has the binomial distribution $\Bin{n}{1-(1-p)^{r}}$ with 
	$\E V(\R)\ge 3npr/5$ (as for large $n$ we have $1-(1-p)^{r}\ge 3pr/5$).  
	Therefore
	\begin{align*}
	\Pra{\exists_{\R\subseteq \W, |\R|\le p^{-1}} V(\R)\le (npr/2) + 1}
	&\le 
	\sum_{r=1}^{p^{-1}}\binom{m}{r}\exp\left(-\psi((1+o(1))5/6)3npr/5\right)\\
	&\le 
	\sum_{r=1}^{p^{-1}}\exp\left(-r\left(0.008np-\ln m\right)\right)=o(1).\\
	\end{align*}	
\end{proof}

\begin{proof}[ of Lemma~\ref{LemProperties2}]

	\noindent{\bf P1}$_*$ This proof is practically the same as the proof of {\bf P1}. 
	
	\noindent{\bf P3}$_*$
	As in the proof of {\bf P3} in the case $np>40$ 
	\begin{align*}
	&\Pra{\exists_{\Set\subseteq \LARGEv_*,|\Set|\le n/\djeden}
		W(\Set)\le 3\cdot 10^{-3}\cdot mp|\Set|
	}\\
	&
	\le 
	\sum_{s=1}^{n/\djeden}
	\binom{n}{s}
	\binom{m}{3\cdot 10^{-3}mps}
	\left(
	\frac
	{3\cdot 10^{-3}mps}
	{m}
	\right)^{6\cdot 10^{-3}mps}\\
	&\le
	\sum_{s=1}^{n/\djeden} n^{s}
	\left(
	e\cdot 3\cdot 10^{-3}ps
	\right)^{3\cdot 10^{-3}mps}\\
	&\le
	\sum_{s=1}^{n/\djeden}
	\exp\left(
	s\left(
	\ln n + 3\cdot 10^{-3}(1+o(1))\ln n (\ln ps + O(1))
	\right)
	\right)
	=o(1).
	\end{align*}
	The above stated inequalities together with \eqref{RownanieVR} imply that if $m<n^{1-\eps}$ then {\whp} for all $\Set\subseteq \LARGEv_*$, $|\Set|\le n/d_1$, and $n$ large
	$$
	N(\Set) = V(\W(\Set))-|\Set|\ge \frac{np|\W(\Set)|}{2}-|\Set|\ge \frac{np}{2}\cdot3\cdot 10^{-3}\cdot mp|\Set|-|\Set|>10^{-3}\djeden|\Set|. 
	$$
	
	\noindent{\bf P4} and {\bf P5} The proofs of  {\bf P4} and {\bf P5} presented in the proof of Lemma~\ref{LemProperties} are also valid in the case $m<n^{1-\eps}$.
	
\end{proof}

\section*{Appendix D}

\begin{proof}[ of Lemma~\ref{LemENDG}]

	Set $\G=\G_n$ and $\G'=\G'_n$. 
	As $\Ham(\G)\subseteq E(\G')$, the execution of HAM on $\G'$ will be the same as on $\G$, {\it i.e.} HAM will start stage $k$ in $\G'$ with $P_k$ (with endpoints $u_0$ and $u_1$) and terminate unsuccessfully in this stage.
	
	Let 
	\begin{equation*}
	\Set_t=\{
	v: v\in \V' \text{ and at stage $k$ there exists a path $Q_s$} \text{ with endpoints $u_0$, $v$ such that $\bar{\delta}(Q_s)=t$}
	\}
	\end{equation*}
	
	First we prove that $\Set_1\neq\emptyset$. $P_k=Q_1$. If $u_1\notin \V'$ then $u_1$ has at least $\delta(\G')$ neighbours on $Q_1$ and at least $\delta(\G')-1$ of them are not connected with $u_1$ by an edge of $Q_1$. Therefore on $Q_1$ there are at  least $\delta(\G')-1\ge 1$ new ends obtained by one rotation from $Q_1$. If $\delta_0=1$ then all of these ends are in $\V'$ (otherwise there would be a vertex with more than $\delta_0=1$ vertices at distance at most $2$ in $\V\setminus\V'$). This implies $\Set_1\neq \emptyset$ in the case $\delta_0=1$.  If $\delta_0 > 1$ ({\it i.e.} $\delta(\G')\ge 1/b_4\ge 3$) then, by (ii), at least 
	$$(1-b_4)\delta(\G')-1\ge \frac{2}{3}\cdot 3 -1 = 1$$ of the ends of the paths obtained by one rotation from $Q_1$ are in $\V'$, {\it i.e.} $\Set_1\neq\emptyset$.
	If $u_1\in \V'$ then $u_1$ has at least $b_3d$ neighbours and at least $b_3d-1$ new ends obtained by a rotation from $Q_1$. Then at least $b_3d-b_4d-1\ge (b_3-b_4)d-1\ge 1$ of these ends are in $\V'$, {\it i.e.} $\Set_1\neq\emptyset$.
	
	Now assume that $1\le |\Set_t|\le n/d$ for some $t$, $1\le t\le T$.
	
	For each $v\in \Set_t$ set a path $Q_{s(v)}$ with endpoints $u_0$ and $v$ such that $\bar{\delta}(Q_{s(v)})=t$. 
	
	Now look at the pairs $(v,u)$, where $u$ is a neighbour of $v$ in $\G'$, {\it i.e.} $u\in \N_{\G'}(v)$. For $\{v,u\}$ which is not an edge of $Q_{s(v)}$, denote by $x(v,u)$  the endpoint of $ROTATE(Q_{s(v)},\{v,u\})$ other than $u_0$. If $x(v,u)\in \V'$ then $x(v,u)\in \Set_{t+1}$.
	
	For any $u\in \N_{\G'}(\Set_t)$ there are at most $2$ vertices $x\in \Set_{t+1}$ such that for some $v\in \Set_t$ we have $x=x(v,u)$  and $\{x,u\}\in E(P_k)$.  
	
	Moreover,
	for any $v$ there is only one 
	$u\in \N_{\G'}(v)$ such that $\{v,u\}\in E(Q_{s(v)})$, at most $b_4d$ vertices $u\in \N_{\G'}(v)$ such that $x(v,u)\notin\V'$ and at most $t$ (where $t=\bar{\delta}(Q_{s(v)})$) vertices $u\in \N_{\G'}(v)$ such that $\{x,u\}\notin E(P_k)$ (since there are at most $t$ edges on $Q_{s(v)}$ which are not in $P_k$).

	Therefore
	
	\begin{align*}
	|\Set_{t+1}|
	&=
	|\{
	x(v,u): v\in \Set_t, u\in \N_{\G'}(v)\text{ and }x(v,u)\in\V'
	\}|\\
	&\ge
	|\{
	x(v,u): v\in \Set_t, u\in \N_{\G'}(v),x(v,u)\in\V',\{u,v\}\notin E(Q_{s(v)}),\{x(v,u),u\}\in E(P_k)
	\}|\\
	&\ge
	|
	\{
	u\in \N_{\G'}(\Set_t):\exists_{v\in \Set_t}x(v,u)\in\V',\{u,v\}\notin E(Q_{s(v)}),\{x(v,u),u\}\in E(P_k)
	\}
	|/2
	\\
	&\ge 
	(N_{\G'}(\Set_t)-(t+1+b_4d)|\Set_t|)/2
	\\
	&\ge 
	((b_3-b_4)d-T-1)|\Set_t|/2
	\\
	&=(1+o(1))(b_3-b_4)d|\Set_t|/2\ge b_5 d |\Set_t|
	\end{align*}  
	for any constant $b_5<(b_3-b_4)/2$.
	
	$\Set_1\neq\emptyset$ and $|\Set_{t+1}|\ge b_5 d |\Set_t|$, therefore for some $\tau\le T-1$ we have $|\Set_{\tau}|\ge n/d$. Applying the same argument as in above calculations for  any $\Set'\subseteq \Set_{\tau}$ such that $|\Set'|=n/d$ we have that $\Set_{\tau+1}\ge b_5d|\Set'|\ge b_5n$. To prove the second part of the lemma, for any $x\in\END{\G'}$, it remains to apply the same arguments replacing $u_0$ by $x$ and $P_k$ by the path $Q_{\sigma}$  with one endpoint $x$ and $\bar{\delta}(Q_{\sigma})\le 2T$.
\end{proof}

\bibliographystyle{abbrvnat}
\bibliography{HAM_short}

\end{document}